\documentclass[12pt]{amsart}

\usepackage{hyperref} \usepackage{amsmath, amsthm, amssymb, amscd}
\usepackage{enumerate}
\usepackage{hyperref}
\usepackage{verbatim}

\usepackage[isbn=false, url=false]{biblatex} 
\DeclareMathOperator{\diam}{diam}

\title{A non-injective Assouad-type theorem with sharp dimension}
\author{Guy C. David}
\address{Department of Mathematical Sciences\\ Ball State University, Muncie, IN 47306}
\email{gcdavid@bsu.edu}
\date{\today}
\thanks{This work was partially supported by the National Science Foundation under Grant No. DMS-2054004.}


\begin{document}
\begin{abstract}
    Lipschitz light maps, defined by Cheeger and Kleiner, are a class of non-injective ``foldings'' between metric spaces that preserve some geometric information. We prove that if a metric space $(X,d)$ has Nagata dimension $n$, then its ``snowflakes'' $(X,d^\epsilon)$ admit Lipschitz light maps to $\mathbb{R}^n$ for all $0<\epsilon<1$. This can be seen as an analog of a well-known theorem of Assouad. We also provide an application to a new variant of conformal dimension.
\end{abstract}

\maketitle

\theoremstyle{plain}
\newtheorem{theorem}{Theorem}
\newtheorem{exercise}{Exercise}
\newtheorem{corollary}[theorem]{Corollary}
\newtheorem{scholium}[theorem]{Scholium}
\newtheorem{claim}[theorem]{Claim}
\newtheorem{observation}[theorem]{Observation}
\newtheorem{lemma}[theorem]{Lemma}
\newtheorem{sublemma}[theorem]{Lemma}
\newtheorem{proposition}[theorem]{Proposition}
\newtheorem{conjecture}{theorem}
\newtheorem{maintheorem}{Theorem}
\newtheorem{maincor}[maintheorem]{Corollary}
\renewcommand{\themaintheorem}{\Alph{maintheorem}}

\theoremstyle{definition}
\newtheorem{fact}[theorem]{Fact}
\newtheorem{example}[theorem]{Example}
\newtheorem{definition}[theorem]{Definition}
\newtheorem{remark}[theorem]{Remark}
\newtheorem{question}[theorem]{Question}

\numberwithin{equation}{section}
\numberwithin{theorem}{section}

\newcommand{\obar}[1]{\overline{#1}}
\newcommand{\haus}[1]{\mathcal{H}^n(#1)}
\newcommand{\prob}{\mathbb{P}}
\newcommand{\Tan}{\text{Tan}}
\newcommand{\WTan}{\text{WTan}}
\newcommand{\CTan}{\text{CTan}}
\newcommand{\CWTan}{\text{CWTan}}
\newcommand{\LIP}{\text{LIP}}
\newcommand{\dist}{\text{dist}}
\newcommand{\cdim}{\text{cdim}}
\newcommand{\bcdim}{\text{bcdim}}
\newcommand{\RR}{\mathbb{R}}
\newcommand{\HH}{\mathcal{H}}
\newcommand{\B}{\mathcal{B}}
\newcommand{\ZZ}{\mathbb{Z}}
\newcommand{\bH}{\mathbb{H}}
\newcommand{\G}{\mathbb{G}}

\section{Introduction}

A basic line of research in metric geometry is the following: Given an abstract metric space, when can one embed, fold, or otherwise map it into a Euclidean space without too much distortion of the geometry? One well-known instance of this question is the bi-Lipschitz embedding problem: Given a metric space, when is there an embedding into some Euclidean space $\mathbb{R}^N$ that preserves all distances up to a constant factor?

There appear to be no simple necessary and sufficient conditions here. A rather obvious necessary condition is that $X$ must have finite \textit{Assouad dimension} $\dim_A$. (See Section \ref{sec:prelim} for a definition.) An important theorem of Assouad says that this condition is sufficient, \textit{if} one is willing to first raise the metric to a power less than one:
\begin{theorem}[Assouad, Proposition 2.6 of \cite{Assouad}]\label{thm:Assouad}
Let $(X,d)$ be a metric space of finite Assouad dimension and $\epsilon\in(0,1)$. Then there is a Euclidean space $\mathbb{R}^N$ and a bi-Lipschitz embedding of the metric space $(X,d^\epsilon)$ into $\mathbb{R}^N$. The dimension $N$ and the distortion of the embedding can be chosen to depend only on $\epsilon$ and the constants in the Assouad dimension of $X$.

\end{theorem}

The metric spaces $(X,d^\epsilon)$, for $\epsilon\in (0,1)$, are called ``snowflakes''. This snowflaking is necessary in Assouad's theorem: there are metric spaces with finite Assouad dimension that have no bi-Lipschitz embedding into any Euclidean space, the most famous being the Heisenberg group \cite[Theorem 7.1]{Semmes}. 

In this paper, we are concerned with a more general class of metric spaces than those in Assouad's theorem. These will be defined via a different notion of dimension, the so-called \textit{Nagata dimension}. The Nagata dimension $\dim_N(X)$ of a metric space can be viewed as a quantification of the purely topological Lebesgue covering dimension $\dim_T(X)$, the  minimal $n\in\mathbb{N}\cup\{0\}$ such that every finite open cover of $X$ admits a refinement of multiplicity at most $n+1$. To define the Nagata dimension, we first declare that a collection of subsets of a metric space $X$ is \textit{$D$-bounded} if each set in the collection has diameter at most $D$. For $s>0$, the \textit{$s$-multiplicity} of a collection of subsets is the minimal $n$ such that every subset of $X$ with diameter at most $s$ intersects at most $n$ members of the collection.

\begin{definition}\label{def:nagata}
The \textit{Nagata dimension} of a metric space $X$, denoted $\dim_N (X)$, is the minimal integer $n$ with the following property: there exists $c>0$ such that, for all $s>0$, $X$ has a $cs$-bounded covering with $s$-multiplicity at most $n+1$.
\end{definition}
The Nagata dimension has turned out to be a very useful quantity to consider for many problems in Lipschitz and quasisymmetric geometry, and a thorough introduction to its properties can be found in \cite{LS}.

In general, \cite[Theorem 1.1]{LeDonneRajala} we have the inequality 
\begin{equation}\label{eq:nagassouad}
\dim_N(X) \leq \dim_A(X) \text{ for all metric spaces } X,
\end{equation}
This inequality is often strict, and it is not even difficult to construct metric spaces of finite Nagata dimension and infinite Assouad dimension.

Assouad's theorem thus no longer generally applies to spaces of finite Nagata dimension, so we may go back to the question at the start of the introduction and ask whether these spaces admit maps to Euclidean space that preserve some geometric information. Our approach is to toss out the injectivity requirement on our mappings, and try to find a way of quantitatively ``folding'', rather than embedding, such spaces into Euclidean space. A class of folding maps that are not necessarily injective but preserve some geometric information at all scales are the so-called ``Lipschitz light'' maps defined by Cheeger and Kleiner in \cite{CK13_inverse}. To define them, we first need to discuss ``$r$-paths'' and related notions:

\begin{definition}
Given $r>0$, a finite sequence $(x_1, \dots, x_k)$ in a metric space is an \textit{$r$-path} if $d(x_i, x_{i+1})\leq r$ for each $i\in\{1, \dots, k-1\}$. 

We say that two points $x,y$ in a metric space $X$ are in the same \textit{$r$-component} of $X$ if there is an $r$-path in $X$ joining them, i.e., an $r$-path $(x_1, \dots, x_k)$ in $X$ with $x=x_1$ and $y=x_k$. For each $r>0$, the notion of $r$-components defines an equivalence relation on $X$.

Lastly, we say that a set is \textit{$r$-connected} if it consists of a single $r$-component, i.e., if every pair of points in it can be joined by an $r$-path.
\end{definition}

\begin{definition}[Cheeger--Kleiner \cite{CK13_inverse}]\label{def:LL2}
A map $f:X\rightarrow Y$ between metric spaces is \textit{Lipschitz light} if there is a constant $C>0$ such that 
\begin{itemize}
\item $f$ is Lipschitz with constant $C$, and
\item for every $r>0$ and every subset $W\subseteq Y$ with $\diam(W)\leq r$, the $r$-components of $f^{-1}(W)$ have diameter at most $Cr$.
\end{itemize}
\end{definition}

Lipschitz light mappings are (topologically) light, so
they cannot collapse any non-trivial continua (compact, connected sets) to points. In fact, Lipschitz light maps preserve more quantitative information. A straightforward rephrasing of Definition \ref{def:LL2} is that $f$ is Lipschitz light if and only if there is a constant $c>0$ such that, for all $r>0$,
$$ c^{-1}\diam(K) - r \leq \diam(f(K)) \leq  c\diam(K) \text{ for all $r$-connected sets } K\subseteq X. $$
Thus, Lipschitz light maps roughly preserve the diameter of coarsely connected sets. In particular, Lipschitz light maps preserve the diameter of continua up to a constant multiplicative factor:
\begin{equation}\label{eq:diam}
 \diam(f(K)) \approx \diam(K) \text{ for all continua } K\subseteq X.
\end{equation}
We therefore consider a Lipschitz light mapping as a way of ``folding'' a metric space that preserves some geometric information.

Our main theorem is then an analog of Assouad's theorem (Theorem \ref{thm:Assouad}) in which Assouad dimension is replaced by Nagata dimension and bi-Lipschitz embeddings are replaced by Lipschitz light maps. Thus, we construct a weaker class of mappings than the embeddings in Assouad's theorem, but for a wider class of spaces. As an important difference, however, we also obtain the sharp dimension of the receiving Euclidean space, and so our result is of interest even for spaces of finite Assouad dimension.

\begin{theorem}\label{thm:main}
Let $(X,d)$ be a metric space of finite Nagata dimension $n\in\mathbb{N}\cup\{0\}$, and let $\epsilon\in(0,1)$. Then there is a Lipschitz light map from $(X,d^\epsilon)$ into $\mathbb{R}^n$. The Lipschitz light constant of this mapping can be chosen to depend only on $n$, $\epsilon$, and the constant $c$ in Definition \ref{def:nagata}.

Moreover, the number $n=\dim_N(X)$ is the minimal integer for which such a map exists.
\end{theorem}

By observation \eqref{eq:diam} above, the Lipschitz light mapping constructed in Theorem \ref{thm:main} has in particular the property that 
\begin{equation}\label{eq:expand}
\diam_{\RR^n}(f(K)) \approx \diam_{(X,d)}(K)^\epsilon \text{ for all continua } K \subseteq (X,d).
\end{equation}

Assouad's Theorem \ref{thm:Assouad} does not hold for $\epsilon=1$, as noted above. Similarly, one cannot take $\epsilon=1$ in Theorem \ref{thm:main}.  Indeed, \cite[Theorem 6.8]{GCD} shows that the Heisenberg group, for example, does not admit a Lipschitz light map into any Euclidean space, while it has Nagata dimension $3$ by \cite[Theorem 4.2]{LeDonneRajala}. Theorem \ref{thm:main} implies that one can ``snowflake and fold'' the Heisenberg group into $\RR^3$. This is rather counter-intuitive: this map from the Heisenberg group to $\RR^3$ \textit{reduces} the overall Hausdorff and Assouad dimensions (from $4$ to $3$) while greatly \textit{expanding} the diameter of all small continua by \eqref{eq:expand}

The parallel between Theorem \ref{thm:main} and Theorem \ref{thm:Assouad} extends to the proofs. As in Assouad's proof, we construct our map by summing certain localized ``bump functions'' over all locations and scales. Differences arise when we consider the way in which these localized functions interact to prove the Lipschitz lightness (in Lemma \ref{lem:bound}). This forces us to work at a sufficiently ``well-separated'' sequence of scales $\{r^j\}_{j\in\mathbb{Z}}$ (see \eqref{eq:rdef}) whereas in Assouad's argument this is not required. We refer to \cite[Chapter 12]{Heinonen} for a presentation of Assouad's argument that influenced our presentation below.

We now present a few corollaries of Theorem \ref{thm:main}.

\subsection{Dimension-theoretic considerations}
A classical result about the topological dimension $\dim_T$ is that (for compact metric spaces), it can also be viewed through (topologically) light maps to Euclidean space. Recall that a continuous map is called light if $f^{-1}(p)$ is totally disconnected for each $p$ in the image of $f$. If $X$ is a compact metric space, then
\begin{equation}\label{eq:toplight}
\dim_T(X) = \min\{ n\geq 0 : \exists f\colon X \rightarrow \RR^n \text{ light}\}.
\end{equation}
(See \cite[Theorems III.6 and III.10]{Nagata}.)

This motivated Cheeger and Kleiner to propose the following new dimension for metric spaces:
\begin{definition}[Cheeger--Kleiner \cite{CK13_inverse}]\label{def:lipdim}
A metric space $X$ has \textit{Lipschitz dimension} $\leq n$ if there is a Lipschitz light map $f:X\rightarrow \mathbb{R}^n$.

We let the \textit{Lipschitz dimension of $X$} be the minimal $n\in\mathbb{N}\cup\{0\}$ such that $X$ has Lipschitz dimension $\leq n$, and denote this by $\dim_L(X)$. If $X$ admits no Lipschitz light map into any Euclidean space, we write $\dim_L(X) = \infty$.
\end{definition}

Cheeger and Kleiner \cite{CK13_inverse} show that spaces of finite Lipschitz dimension can be represented as certain inverse limits, and furthermore that spaces of Lipschitz dimension $1$ admit bi-Lipschitz embeddings into the Banach space $L_1$. Further properties of Lipschitz dimension are explored in \cite{GCD, Freeman, FG1, FG2}.

It is already perhaps surprising, in view of \eqref{eq:toplight} in the topological realm, that the Lipschitz and Nagata dimensions differ. (For the Heisenberg group $H$, $\dim_N(H)=3$ and $\dim_L(H)=\infty$, as noted above.) However, we do always have the inequalities
\begin{equation}\label{eq:topnaglip}
\dim_T(X) \leq \dim_N(X) \leq \dim_L(X)
\end{equation}
by \cite[Theorem 2.2]{LS} and \cite[Corollary 3.5]{GCD}.

In view of the results of \cite{CK13_inverse} and the general interest in bi-Lipschitz invariants of metric spaces, it is worth understanding further the metric spaces of finite Lipschitz dimension. A rephrasing of our main theorem says the following:
\begin{corollary}\label{thm:snowflake}
Suppose that a metric space $(X,d)$ has Nagata dimension $n\in\mathbb{N}\cup\{0,\infty\}$. Then for each $\epsilon\in (0,1)$, the metric space $(X,d^\epsilon)$ has Lipschitz dimension $n$.
\end{corollary}
It follows from this corollary and \cite[Theorem 1.7]{CK13_inverse} that each snowflake of a metric space of Nagata dimension $1$ admits a bi-Lipschitz embedding to $L_1(Z,\mu)$ for some measure space $(Z,\mu)$.  

As another direct corollary of Theorem \ref{thm:snowflake}, we answer \cite[Question 8.5]{GCD}:
\begin{corollary}\label{cor:increase}
Snowflaking cannot increase the Lipschitz dimension of a metric space. In other words, $\dim_L(X,d^\epsilon) \leq \dim_L(X,d)$ for every metric space $(X,d)$ and every $\epsilon\in (0,1)$.
\end{corollary}
\begin{proof}
Let $X$ be a metric space and $\epsilon\in(0,1)$. By Corollary \ref{thm:snowflake} and \eqref{eq:topnaglip}, we have
$$ \dim_L(X,d^\epsilon) = \dim_N(X,d) \leq \dim_L(X,d).$$
\end{proof}

\subsection{``Branched conformal'' dimension}
We now discuss an application of Theorem \ref{thm:main} to a variation on the well-studied notion of \textit{conformal dimension}.

\subsubsection{Quasisymmetries and branched quasisymmetries}
The snowflaking transformation $(X,d) \mapsto (X,d^\epsilon)$ is a special case of a class of mappings called quasisymmetric mappings. These arose in classical conformal mapping theory and now play a major role in analysis on metric spaces, hyperbolic geometry, and geometric group theory. 

\begin{definition}\label{def:bqs}
An embedding $f\colon X \rightarrow Y$ between metric spaces is called \textit{quasisymmetric} (or a \textit{quasisymmetry}) if there is a homeomorphism $\eta\colon [0,\infty) \rightarrow [0,\infty)$ such that
$$ d(f(x), f(a)) \leq \eta\left( \frac{d(x,a)}{d(x,b)} \right) d(f(x),f(b))$$
for all triples $a,b,x$ of distinct points in $X$.

If the embedding is surjective, we say that $X$ and $Y$ are quasisymmetric.
\end{definition}
In other words, quasisymmetric maps are homeomorphisms that distort \textit{ratios} of distances in a controlled manner. There is now a vast theory of quasisymmetric mappings in general metric spaces; we refer the interested reader to \cite{Heinonen} for background, but we will not really need it here. The most salient facts for us are that Nagata dimension is a quasisymmetric invariant (see \cite[Theorem 1.2]{LS}), but Assouad and Lipschitz dimension are not. (See \cite{MackayTyson} for the former and \cite[Corollary 8.4]{GCD} for the latter.)

More recently, there has been interest in analogs of quasisymmetric mappings that need not be injective, somewhat like the relationship between holomorphic functions and conformal mappings in classical complex analysis. To this end, Guo and Williams \cite{GuoWilliams} defined and studied ``branched quasisymmetric'' mappings. To define them, we first recall that a \textit{continuum} in a metric space is a compact, connected set. If the continuum has at least two points, we call it \textit{non-trivial}. 

\begin{definition}
A continuous mapping $f\colon X\rightarrow Y$ is called \textit{branched quasisymmetric} (or a \textit{branched quasisymmetry}) if there is a homeomorphism $\eta\colon [0,\infty) \rightarrow [0,\infty)$ such that
$$\diam(f(E)) \leq \eta\left(\frac{\diam(E)}{\diam(E')}\right) \diam(f(E)),$$
for all pairs of intersecting non-trivial continua $E$ and $E'$ in $X$.
\end{definition}
Thus, branched quasisymmetries are continuous maps that distort ratios of diameters of intersecting continua in a controlled manner, but they need not be injective. (We remark that in \cite{GuoWilliams}, but not in \cite{LindquistPankka}, branched quasisymmetries are required also to be discrete open mappings. We do not require this.) A simple argument shows that quasisymmetric embeddings are automatically branched quasisymmetric.

The definition of branched quasisymmetry is rather vacuous if $X$ has no continua. Thus, the natural setting for studying branched quasisymmetries is not arbitrary metric spaces, but those that are equipped with many continua: the so-called ``bounded turning'' (or ``linearly connected'') spaces. A metric space is called \textit{bounded turning} if every pair of points can be joined by a continuum whose diameter is comparable to the distance between the points.

\subsubsection{Conformal and ``branched conformal'' dimension}
Assouad dimension is not a quasisymmetric invariant, but the \textit{conformal Assouad dimension} is. This quantity for a metric space $X$ is defined as
\begin{equation}\label{eq:confdim}
\text{confdim}_A(X) = \inf\{ \dim_A(Y): Y\text{ is quasisymmetric  to } X\}.
\end{equation}
Conformal Assouad dimension (along with its variations) has played a major role in geometric group theory and hyperbolic geometry; we refer to \cite{MackayTyson} for an overview. Computing this quantity even for well-known spaces like the Sierpi\'nski carpet remains an open problem, and in many cases the infimum is not actually achieved.

A natural question is then, what if we allow \textit{branched} quasisymmetric images $Y$ of $X$ in \eqref{eq:confdim}? We would obtain another quasisymmetrically invariant quantity, potentially smaller than conformal Assouad dimension. Is it computable? Is the infimum achieved?

In fact, we show that for compact, bounded turning spaces this construction exactly recovers the Nagata dimension. In particular, in stark contrast to (non-branched) conformal dimension, we obtain an integer-valued quantity in which the infimum is always achieved. This gives a new characterization of the Nagata dimension.

\begin{corollary}\label{cor:bcdim}
Let $X$ be a compact, bounded turning metric space. Then
$$ \dim_N(X) = \inf\{ \dim_A(Y): \text{there is a non-constant branched quasisymmetry of } X \text{ onto } Y\}.$$
This includes the statement that if one of these numbers is $\infty$, then the other is as well.

Moreover, the infimum above is always achieved if $\dim_N(X)<\infty$.
\end{corollary}

\subsection{Organization of the paper}
Section \ref{sec:prelim} contains basic definitions and notation, section \ref{sec:mainproof} contains the proof of Theorem \ref{thm:main}, and section \ref{sec:bqs} contains additional background and the proof of Corollary \ref{cor:bcdim}.

\section{Preliminaries}\label{sec:prelim}

\subsection{Basics}
We write $(X,d)$ for a metric space, but often just $X$ with the metric $d$ understood from context. In such a case, we often below write $X^\epsilon$ for the metric space $(X,d^\epsilon)$ (which is in fact a metric space if $\epsilon\in(0,1)$. 

An open ball in $X$ centered at $x\in X$ of radius $r$ is denoted $B(x,r)$, and the associated closed ball by $\overline{B}(x,r)$. 

The distance between a point $x\in X$ and a set $E\subseteq X$ is
$$ \dist(x,E) = \inf\{ d(x,y) : y\in E\}$$
and the diameter of the set $E$ is
$$ \diam(E) = \sup\{ d(x,y): x,y\in E\}.$$
In cases where we want to clarify which metric space we are discussing, we may include it as a subscript, e.g., writing $d_X(a,b)$ or $\diam_X(E)$. Similarly, we write $\diam^\epsilon(E)$ for the diameter of $E$ in the metric space $(X,d^\epsilon)$, which of course is simply $(\diam(E))^\epsilon$.

We use the notation $A \lesssim B$ to indicate that $A \leq c B$ for some constant $c$ that may change from line to line. If we wish to denote what quantities $\alpha, \beta, \dots$ the constant $c$ depends on, we write $A \lesssim_{\alpha,\beta,\dots} B$. We write $A\approx B$ to mean $A\lesssim B$ and $B\lesssim A$.
\subsection{Summary of dimensions}
Three notions of dimension for a metric space $X$ are used in this paper: The Nagata dimension $\dim_N(X)$ (see Definition \ref{def:nagata}), the Lipschitz dimension $\dim_L(X)$ (see Definition \ref{def:lipdim}), and the Assouad dimension $\dim_A(X)$. 

Although we do not use the definition of Assouad dimension below, we include the definition for context: The Assouad dimension of $X$ is the infimum of all numbers $\beta>0$ such that, for some constant $C>0$, every set of diameter $d$ in $X$ can be covered by at most $C\epsilon^{-\beta}$ sets of diameter at most $\epsilon d$. See \cite[Chapter 10]{Heinonen} for more discussion. Unlike the Nagata and Lipschitz dimensions, the Assouad dimension may not be an integer.

\section{Proof of Theorem \ref{thm:main}}\label{sec:mainproof}
In this section, we prove the main result of the paper, Theorem \ref{thm:main}. Fix a metric space $X$ of Nagata dimension $n\in\mathbb{N}\cup\{0\}$ and $\epsilon\in (0,1)$.

We first quickly show the ``moreover...'' statement of the theorem: that $X^\epsilon$ does not admit a Lipschitz light map to $\mathbb{R}^k$ for any $k<n$. If there were such a map, then we would have $\dim_L(X^\epsilon)\leq k$. However, by \cite[Corollary 3.5]{GCD} and the quasisymmetric invariance of Nagata dimension, we would then have
\begin{equation}\label{eq:infinite}
 k\geq \dim_L(X^\epsilon) \geq \dim_N(X^\epsilon) = \dim_N(X) = n,
\end{equation}
a contradiction.

In the rest of this section, we focus on proving the first part of Theorem \ref{thm:main} by constructing a Lipschitz light map from $X^\epsilon$ to $\RR^n$.

Let $C$ be the associated Nagata dimension constant from Definition \ref{def:nagata}. The following proposition of Lang and Schlichenmaier provides a useful collection of coverings of $X$; we state only the conclusions of their result that we will need later on.

\begin{proposition}[Proposition 4.1 of \cite{LS}]\label{prop:cover}.
There are constants $\hat{c}$ and $\hat{r}_0$ (depending only on $n$ and $C$) such that for each $r\geq \hat{r}_0$, there is a sequence of coverings $\hat{\mathcal{B}}^j$ of $X$ ($j\in\ZZ$) such that:
\begin{enumerate}[(i)]
\item For each $j\in\ZZ$, $\hat{\B}^j = \cup_{k=0}^n \hat{\B}^j_k$, where each $\hat{\B}^j_k$ is a $\hat{c}r^j$-bounded family of $r^j$-multiplicity at most $1$.
\item For each $j\in\ZZ$ and $x\in X$, there is a $\hat{B}\in\hat{\B}^j$ that contains the closed ball $\overline{B}(x,r^j)$. 
\end{enumerate}

\end{proposition}
For our purposes, it is more convenient to work with open sets (and to have a constant $\geq 1$), so we make a minor adjustment to Proposition \ref{prop:cover}.

\begin{lemma}\label{lem:cover}
There are constants $c\geq 1$ and $r_0$ (depending only on $n$ and $C$) such that for each $r\geq r_0$, there is a sequence of coverings $\mathcal{B}^j$ of $X$ ($j\in\ZZ$) such that:
\begin{enumerate}[(i)]
\item For each $j\in\ZZ$, $\B^j = \cup_{k=0}^n \B^j_k$, where each $\B^j_k$ is a $cr^j$-bounded family of $\frac{1}{2}r^j$-multiplicity at most $1$.
\item For each $j\in\ZZ$ and $x\in X$, there is a $B\in\B^j$ that contains the closed ball $\overline{B}(x,r^j)$. 
\item For each $j\in\ZZ$ and $B\in\B^j$, the set $B$ is open.
\end{enumerate}
\end{lemma}
\begin{proof}
Apply Proposition \ref{prop:cover}. For each $j\in\mathbb{Z}$, let $\B^j$ be the collection of open $\frac{1}{4}r^j$-neighborhoods of elements of $\hat{\B}^j$. The properties in the lemma then follow easily, with $r_0 = \hat{r}_0$ and $c=\hat{c} + \frac{1}{2}$. If necessary, we may replace $c$ by $\max\{c,1\}$ without changing the conclusion of the lemma, in order to ensure that $c\geq 1$.
\end{proof}

We now fix $r$ sufficiently large so that Lemma \ref{lem:cover} holds and in addition so that 
\begin{equation}\label{eq:rdef}
\frac{2}{r^\epsilon - 1} + \frac{4c}{r^{1-\epsilon} -1} < 1 - \frac{1}{2^\epsilon},
\end{equation}
where $c$ is the constant from Lemma \ref{lem:cover}. We apply Lemma \ref{lem:cover} and obtain coverings $\{\B^j\}_{j\in\ZZ}$ of $X$ that will be fixed for the remainder of the proof.

Given $j\in\ZZ$ and $B\in \B^j$, define $\psi_B\colon X \rightarrow \RR$ by 
$$ \psi_B(x) = \min\left\{1, r^{-j} \dist(x, B^c)\right\}.$$

We have the following basic properties of $\psi_B$.
\begin{lemma}\label{lem:psi}
The functions $\psi_B$ have the following properties:
\begin{enumerate}[(i)]
\item $0 \leq \psi_B(x) \leq 1$ for all $B\in \B, x\in X$.
\item $\psi_B(x) > 0$ if and only if $x\in B$. 
\item If $B\in\B^j$, then $\psi_B$ is Lipschitz with constant at most $r^{-j}$.
\item If $x\in X$ and $j\in \ZZ$, then there is a $k\in\{0,\dots, n\}$ and a $B\in \B^j_k$ such that $\psi_B(x)= 1$. 
\end{enumerate}
\end{lemma}
\begin{proof}
Item (i) is immediate. For item (ii), if $\psi_B(X)>0$, then $x\notin B^c$, so $x\in B$. If $x\in B$, then since $B$ is open, $\psi_B(x)>0$. For item (iii), distance functions are $1$-Lipschitz by the triangle inequality, and the minimum of two $L$-Lipschitz functions is $L$-Lipschitz. Item (iv) follows immediately from the definition of $\psi_B$ and property (ii) of Lemma \ref{lem:cover}.
\end{proof}

Let $e_0 = 0\in \RR^n$ and let $\{e_1, \dots, e_n\}$ be an orthonormal basis of $\RR^n$.  For each $j\in \ZZ$, set
$$ \phi^j(x) = \sum_{k=0}^n \left(\sum_{B\in \B^j_k} \psi_B(x)\right)e_k.$$

\begin{lemma}\label{lem:phi}
The functions $\phi^j$ have the following properties, for each $j\in\ZZ$:
\begin{enumerate}[(i)]
\item $|\phi^j(x)| \leq n+1$ for all $x\in X$.
\item As a function from $X$ to $\mathbb{R}^n$, $\phi^j$ is Lipschitz with constant $2(n+1)r^{-j}$.
\end{enumerate}
\end{lemma}
\begin{proof}
By Lemma \ref{lem:cover}(i), a given point can be in $B$ for at most $n+1$ different sets $B\in\B^j$. It follows from this and Lemma \ref{lem:psi}(ii) that at most $(n+1)$ summands in the rightmost sum of the chain 
$$|\phi^j(x)| \leq \sum_{k=0}^n \sum_{B\in \B^j_k} |\psi_B(x)| = \sum_{B\in \B^j} |\psi_B(x)|$$
are non-zero. By Lemma \ref{lem:psi}(i), each summand is bounded by $1$, so $|\phi^j(x)|\leq n+1$, proving (i). 

Similarly, given $x,y\in X$, at most $2(n+1)$ of the summands in the rightmost sum of the chain 
$$|\phi^j(x) - \phi^j(y)| \leq \sum_{k=0}^n \left|\sum_{B\in \B^j_k} \psi_B(x) - \psi_B(y)\right| \leq \sum_{B\in \B^j} |\psi_B(x) - \psi_B(y)|$$
are non-zero. By Lemma \ref{lem:psi}(iii), each summand is bounded above by $r^{-j}d(x,y)$. Therefore
$$|\phi^j(x) - \phi^j(y)| \leq 2(n+1)r^{-j}d(x,y),$$
proving (ii).

\end{proof}

Fix $x_0\in X$. Define $f\colon X \rightarrow \RR^n$ by
\begin{equation}\label{eq:fdef}
 f(x) = \sum_{j\in\ZZ} r^{j\epsilon} (\phi^j(x) - \phi^j(x_0))
\end{equation}
In the remainder of the section, we prove that $f$ is Lipschitz light when viewed as a map from $X^\epsilon$ to $\RR^n$. First, we show that it is Lipschitz.

\begin{lemma}\label{lem:lipschitz}
The sum in \eqref{eq:fdef} converges absolutely for each $x\in X$ and defines a Lipschitz map $f\colon X^\epsilon \rightarrow \RR^n$. The Lipschitz constant of $f$ can be bounded above depending only on $C$, $\epsilon$, and $n$.
\end{lemma}
\begin{proof}
Assuming the convergence statement for a moment, let us prove that $f$ is Lipschitz.

Let $x,y\in X$. Choose $j_0\in \ZZ$ such that $r^{j_0} \leq d(x,y) < r^{j_0+1}$. Then, using both parts of Lemma \ref{lem:phi}, we have
\begin{align*}
|f(x)-f(y)| &\leq \sum_{j\in\ZZ} r^{j\epsilon} |\phi^j(x) - \phi^j(y)|\\
&\leq (n+1)\sum_{j\leq j_0} r^{j\epsilon} + 2(n+1)\sum_{j>j_0} r^{j\epsilon} r^{-j} d(x,y)\\
&\lesssim_{r,n} r^{j_0\epsilon} + r^{j_0(\epsilon-1)}d(x,y)\\
&\lesssim_{r,n} r^{j_0\epsilon} + r^{j_0\epsilon}\\
&\lesssim_{r,n} d(x,y)^\epsilon.
\end{align*}
The choice of $r$ was made above depending only on $\epsilon$ and $c$, which itself depends only on $C$ and $n$.

The proof that the sum in \eqref{eq:fdef} converges absolutely is essentially identical, just with $y$ replaced by $x_0$.
\end{proof}

We now work to prove that $f$ is Lipschitz light. The following is the main technical lemma.
\begin{lemma}\label{lem:bound}
Suppose that $x,y\in X$, $s>0$, $j_0\in\ZZ$, and $k\in\{1, \dots, n\}$ have the following properties:
\begin{enumerate}[(i)]
\item $|f(x)-f(y)|\leq s$,
\item $d(x,y) \leq 2c r^{j_0}$,
\item $r^{j_0\epsilon} > 2^\epsilon s$, and
\item $\sum_{B\in\B^{j_0}_k} \psi_B(x) = 1$.
\end{enumerate}
Then
$$y\in \bigcup_{B\in \B^{j_0}_k} B.$$
\end{lemma}
\begin{remark}
The purpose of this lemma is as follows: Given a scale $j_0$ and two points $x,y$, the lemma provides sufficient conditions to conclude that $x$ and $y$ lie in sets from $\mathcal{B}^{j_0}$ of the same ``color'', i.e., from the same family $\mathcal{B}^{j_0}_k$ for some $k$.

The reader may wonder about the variable $s$, which appears in the hypotheses but not the conclusions. This is merely to make the application of this lemma more transparent in the proof of Theorem \ref{thm:main} below.
\end{remark}

\begin{proof}[Proof of Lemma \ref{lem:bound}]
Note that the $k$th coordinate of $f(x)-f(y)$ can be written as
$$ \sum_{j\in\ZZ} r^{j\epsilon} \sum_{B\in \B^j_k} (\psi_B(x) - \psi_B(y)).$$

Suppose that conditions (i) through (iv) hold but that $y\notin \cup_{B\in \B^{j_0}_k} B$. By Lemma \ref{lem:psi}, it follows that $\psi_B(y) =0$ for each $B\in\B^{j_0}_k$. Then,
\begin{align}
s &\geq |f(x)-f(y)|\nonumber\\
&\geq \left|\sum_{j\in\ZZ}r^{j\epsilon} \sum_{B\in \B^j_k} (\psi_B(x) - \psi_B(y)) \right|\nonumber\\
&\geq r^{j_0 \epsilon} - \left|\sum_{j<j_0}r^{j\epsilon} \sum_{B\in \B^j_k} (\psi_B(x) - \psi_B(y))\right| - \left|\sum_{j>j_0}r^{j\epsilon} \sum_{B\in \B^j_k} (\psi_B(x) - \psi_B(y)\right|\nonumber\\
&\geq r^{j_0 \epsilon} - \sum_{j<j_0}r^{j\epsilon} \sum_{B\in \B^j_k} |\psi_B(x) - \psi_B(y))| - \sum_{j>j_0}r^{j\epsilon} \sum_{B\in \B^j_k} |\psi_B(x) - \psi_B(y)|| \label{eq:triangle}
\end{align}

For each $j<j_0$ in the sum in the middle of \eqref{eq:triangle}, there are at most two choices of $B\in\mathcal{B}^j_k$ for which the difference $|\psi_B(x) - \psi_B(y)|$ is non-zero. (This is because each of $x,y$ can be contained in at most one $B\in\mathcal{B}^j_k$, and Lemma \ref{lem:psi}(ii).) For those $B\in\mathcal{B}^j_k$ where this difference is non-zero, it is at most $1$ by Lemma \ref{lem:psi}(i).

Similarly, for each $j>j_0$ in the final sum of \eqref{eq:triangle}, there are at most two choices of $B\in\mathcal{B}^j_k$ for which the difference $|\psi_B(x) - \psi_B(y)|$ is non-zero.  For those $B\in\mathcal{B}^j_k$ where this difference is non-zero, it is at most $r^{-j}d(x,y)$ by Lemma \ref{lem:psi}(iii).

Thus, we have:
\begin{align*}
s&\geq r^{j_0 \epsilon} - 2\sum_{j<j_0}r^{j\epsilon} - 2\sum_{j>j_0} r^{j\epsilon}r^{-j}d(x,y)\\
&= r^{j_0\epsilon} - \frac{2r^{j_0\epsilon}}{r^\epsilon-1} - \frac{2r^{j_0(\epsilon-1)}}{r^{1-\epsilon} -1}d(x,y)\\
&\geq r^{j_0\epsilon} - \frac{2r^{j_0\epsilon}}{r^\epsilon-1} - \frac{4c r^{j_0\epsilon}}{r^{1-\epsilon} -1}\\
&= r^{j_0\epsilon}\left( 1 - \frac{2}{r^\epsilon - 1} - \frac{4c}{r^{1-\epsilon} -1} \right)
\end{align*}

By our choice of $r$ in \eqref{eq:rdef}, the quantity in parentheses above is at least $\frac{1}{2^\epsilon}$. It follows that $r^{j_0\epsilon}\leq 2^\epsilon s$, contradicting assumption (iii) of this lemma. 

\end{proof}

We now complete the proof of Theorem \ref{thm:main} by showing that the map $f\colon X^\epsilon \rightarrow \RR^n$ constructed in \eqref{eq:fdef} is Lipschitz light, with constants depending only on $\epsilon$, $n$, and $C$.
\begin{proof}[Proof of Theorem \ref{thm:main}]
The map $f$ is already Lipschitz by Lemma \ref{lem:lipschitz}. To prove that $f$ is Lipschitz light, it is enough to show that any $s$-path $P$ in $X^\epsilon$ with $\diam(f(P))\leq s$ must have $\diam^\epsilon(P) \lesssim_{\epsilon,n,C} s$. Thus, let $s>0$ and let $P$ be an $s$-path in $X^\epsilon$ with $\diam(f(P))\leq s$. We will show that
\begin{equation}\label{eq:diambound}
\diam(P) \leq 10cr s^{\frac{1}{\epsilon}},
\end{equation}
which of course implies that 
$$ \diam^\epsilon(P) \leq (10cr)^\epsilon s.$$
As $r$ and $c$ depend only on $\epsilon$, $n$, and $C$, this will prove the lemma.

Suppose to the contrary that $\diam(P)> 10cr s^{\frac{1}{\epsilon}}$. Choose $j_0\in\ZZ$ so that
\begin{equation}\label{eq:assumptioniii}
 2s^{\frac{1}{\epsilon}} < r^{j_0} \leq 2r s^{\frac{1}{\epsilon}} < \frac{1}{5c}\diam(P).
\end{equation}

First of all, we note that $P$ cannot be completely contained in $\cup_{B\in B^{j_0}_0} B$: By Lemma \ref{lem:cover}(i), distinct sets in $B^{j_0}_0$ are separated by $d^\epsilon$-distance at least $\frac{1}{2^\epsilon}r^{j_0\epsilon}>s$, so if the $s$-path $P$ were contained in this union it would have to be contained in a single set $B\in B^{j_0}_0$. But in that case, the $d$-diameter of $P$ would be bounded by $\diam(B)\leq c r^{j_0} \leq 2cr s^{\frac{1}{\epsilon}} < \diam(P)$, a contradiction.

Therefore, there is a point $x\in P$ that is not contained in $\cup_{B\in B^{j_0}_0} B$. By Lemma \ref{lem:psi}(iv), there is a $k\in\{1, \dots, n\}$ and a $B\in \B^{j_0}_k$ with $\psi_B(x)=1$. Since $x$ can be in at most one element of $\B^{j_0}_k$, we have
\begin{equation}\label{eq:assumptioniv}
\sum_{B\in\B^{j_0}_k} \psi_B(x) = 1.
\end{equation}

Now, $P$ must contain a point $z$ such that
$$ d(z,x) \geq \frac12 \diam(P) > 2cr^{j_0}.$$

By truncating $P$ (and reversing the order if necessary), we can find an $s$-path $Q = (x_1, x_2, \dots, x_m)$ (in the $d^\epsilon$-distance) with $x_1=x$, $x_m=z$, and each $x_i$ contained in $P$.
Let $x_b$ denote the first point in $Q$ such that $d(x_b,x)>2cr^{j_0}$.

Let $R$ be the sub-path $(x_{1}, x_{2}, \dots, x_{b-1})$. This is an $s$-path in the $d^\epsilon$-distance with the property that
\begin{equation}\label{eq:assumptionii}
 d(x_i,x) \leq 2cr^{j_0} \text{ for each } x_i\in R.
\end{equation}

Fix an arbitrary $i\in\{1,2,\dots, b-1\}$. We now apply Lemma \ref{lem:bound} with $x,j_0, k,s$ as above and $y=x_i\in R$. Assumption (i) of that lemma is satisfied because $x_i, x\in P$ and $\diam(f(P))\leq s$.  Assumptions (ii) through (iv) are verified in \eqref{eq:assumptionii}, \eqref{eq:assumptioniii}, and \eqref{eq:assumptioniv}, respectively. Therefore, each $x_i\in R$ is contained in $\cup_{B\in \B^{j_0}_k} B$. The sets in $\B^{j_0}_k$ are $\frac{1}{2^\epsilon} r^{j_0\epsilon}$-separated in the $d^\epsilon$-distance, and $R$ is an $s$-path in the $d^\epsilon$-distance with $s<\frac{1}{2^\epsilon}r^{j_0\epsilon}$ by \eqref{eq:assumptioniii}. Therefore, $R$ must be completely contained in a single set $B\in\B^{j_0}_k$. Therefore
$$ \diam(R) \leq \diam(B) \leq cr^{j_0}.$$ 

On the other hand,
$$ \diam(R) \geq d(x_{b-1},x_1) \geq d(x_b,x) - s^{\frac{1}{\epsilon}} \geq 2cr^{j_0} - \frac{1}{2}r^{j_0} \geq (2c-\frac{1}{2})r^{j_0} > cr^{j_0},$$
recalling that $c\geq 1$ in Lemma \ref{lem:cover}.

This is a contradiction. We have therefore proven \eqref{eq:diambound} and hence that $f$ is Lipschitz light with quantitative control on the constants.
\end{proof}

\section{Branched quasisymmetries}\label{sec:bqs}
In this section, we first present some basic facts about branched quasisymmetries due to Guo--Williams \cite{GuoWilliams}, and then prove Corollary \ref{cor:bcdim}.

\subsection{The pullback metric of Guo--Williams}\label{subsec:bqs}
In their study of branched quasisymmetries in \cite{GuoWilliams}, Guo and Williams showed that these mappings factor in a useful way by using a device called the pullback metric. As our assumptions and statements are slightly different than theirs, and we do not need all the notation they present, we present all the facts and proofs we need here. However, we emphasize that all the ideas in subsection \ref{subsec:bqs} are due originally to Guo and Williams, and we present it merely as necessary background.

Throughout subsection \ref{subsec:bqs}, we fix $X$ as a compact, bounded turning metric space and $f\colon X\rightarrow Y$ be a non-constant branched quasisymmetry. Define the \textit{pullback metric} on $X$ by
\begin{equation}\label{eq:df}
 d_f(x,y) = \inf\{ \diam(f(K)): K \text{ a continuum containing } x \text{ and } y\}.
\end{equation}
We write $\diam_f(E)$ for the diameter of a set $E$ in the metric $d_f$.

\begin{remark}\label{rmk:inf}
Under these assumptions, the infimum in \eqref{eq:df} is always achieved: Given $x,y\in X$, let $K_n$ be a sequence of continua with 
$$ \diam(f(K_n)) \rightarrow d_f(x,y).$$
Since $X$ is compact and $f$ is continuous, a subsequence of these continua converges in the Hausdorff sense to a continuum $K$ satisfying
\begin{equation}\label{eq:inf}
\diam(f(K)) = d_f(x,y).
\end{equation}
\end{remark}

\begin{lemma}
This $d_f$ is a metric on $X$ that is topologically equivalent to the original metric $d_X$.
\end{lemma}
\begin{proof}
First of all, note that $f$ cannot collapse any non-trivial continuum to a point. Indeed, if it did then the bounded turning assumption and Definition \ref{def:bqs} would force $f$ to collapse all of $X$ to a single point, but we assumed that $f$ was non-constant.

It follows from this and Remark \ref{rmk:inf} that $d_f$ is positive definite. The symmetry and triangle inequality are immediate from the definition of $d_f$, and so $d_f$ is a metric on $X$.

Now suppose that $x_n \rightarrow x$ in $d_X$. Since $X$ is bounded turning, there are continua $K_n$ containing both $x_n$ and $x$ with 
$$ \diam(K_n) \lesssim d(x_n,x) \rightarrow 0 \text{ as } n\rightarrow\infty.$$
Let $K$ be any non-trivial fixed continuum containing $x$, and recall that by our first paragraph $f(K)$ is a non-trivial continuum as well. It follows from Definition \ref{def:bqs} that
$$ \diam(f(K_n)) \leq \diam(f(K))\cdot \eta(\diam(K_n)/\diam(K)) \rightarrow 0 \text{ as } n\rightarrow\infty.
$$
Hence $d_f(x_n,x)\rightarrow 0$.

Conversely, suppose that $x_n\rightarrow x$ in $d_f$. Then we have a sequence $E_n$ of continua containing $x_n$ and $x$, with $\diam(f(E_n)) \rightarrow 0$. Let $K$ be a non-trivial continuum containing $x$, as above. We then have
\begin{equation}\label{eq:lowerdiam}
 0<\diam(f(K))\leq \diam(f(E_n)) \eta(\diam(K)/\diam(E_n)). 
\end{equation}
We know that $\diam(f(E_n))\rightarrow 0$, so for the expression on the right-hand side of \eqref{eq:lowerdiam} to be bounded away from zero, we must have $\eta(\diam(K)/\diam(E_n))\rightarrow\infty$, or equivalently, $\diam(E_n)\rightarrow 0$. It follows that
$$ d(x_n, x)\leq \diam(E_n)\rightarrow 0.$$
\end{proof}

\begin{lemma}\label{lem:factor}
The map $f:(X,d_f)\rightarrow Y$ is $1$-Lipschitz and furthermore satisfies
\begin{equation}\label{eq:diam2}
 \diam(f(J)) = \diam_f(J)
\end{equation}
for all continua $J$ in $(X,d_f)$.

Moreover, the space $(X,d_f)$ is bounded turning with constant $1$.
\end{lemma}
This is a slight modification of \cite[Lemma 5.7]{GuoWilliams}.
\begin{proof}

Consider $x,y\in X$. Let $K$ be as in \eqref{eq:inf}, so that
$$ d_f(x,y) = \diam(f(K)).$$
In that case,
$$ d(f(x),f(y)) \leq \diam(f(K)) = d_f(x,y),$$
and so $f$ is $1$-Lipschitz.

Continuing with $x$, $y$, and $K$ as above, note that if $p,q\in K$, then
$$ d_f(p,q) \leq \diam(f(K)) = d_f(x,y).$$
This proves that $\diam_f(K)\leq d_f(x,y)$, and so $(X,d_f)$ is $1$-bounded turning.

Now consider any continuum $J\subseteq (X,d_f)$. Since $f$ is $1$-Lipschitz on $(X,d_f)$, we have $\diam(f(J)) \leq \diam_f(J)$. On the other hand, if $x$ and $y$ are in $J$, then
$$ d_f(x,y) \leq \diam(f(J)),$$
and so $\diam_f(J) = \diam(f(J))$.

\end{proof}

\begin{lemma}\label{lem:bqsqs}
The identity map is a quasisymmetric homeomorphism from $(X,d_X)$ to $(X,d_f)$.
\end{lemma}
This is contained in \cite[Propositions 6.47 and 6.48]{GuoWilliams}.
\begin{proof}
We first observe that the identity map from $(X,d_X)$ to $(X,d_f)$ is a branched quasisymmetry. Indeed, if $E$ and $E'$ are intersecting continua in $X$, then by the previous lemma and the fact that $f\colon (X, d_X) \rightarrow Y$ is a branched quasisymmetry, we have
$$ \diam_f(E) = \diam(f(E)) \leq \diam(f(E')) \eta\left(\frac{\diam_X(E)}{\diam_X(E')}\right) = \diam_f(E') \eta\left(\frac{\diam_X(E)}{\diam_X(E')}\right).$$

Next we argue as in Proposition 6.48 of \cite{GuoWilliams} to show that the identity is in fact a quasisymmetric homeomorphism. Let $x,y,z\in X$ be distinct points. As $(X,d_X)$ is $C$-bounded turning and $(X,d_f)$ is $1$-bounded turning, we may choose continua $E$ containing $x$ and $y$ and $E'$ containing $x$ and $z$ such that $\diam_X(E) \leq C d_X(x,y)$ and $\diam_f(E')= d_f(x,z)$.

Then
$$ \frac{d_f(x,y)}{d_f(x,z)} \leq \frac{\diam_f(E)}{\diam_f(E')} \leq \eta\left(\frac{\diam_X(E)}{\diam_X(E')}\right) \leq \eta\left(C\frac{d(x,y)}{d(x,z)}\right).$$
Thus, the identity from $(X,d_X)$ to $(X,d_f)$ is $\tilde{\eta}$-quasisymmetric, where $\tilde{\eta}(t) = \eta(Ct)$.

\end{proof}

We summarize the lemmas above as follows: if $X$ is compact and bounded turning and $f\colon X \rightarrow Y$ is branched quasisymmetric, then $f$ factors as
\begin{equation}\label{eq:factor}
 X \rightarrow (X, d_f) \rightarrow Y,
\end{equation}
where the first map is a quasisymmetric homeomorphism and the second map preserves the diameter of all continua.

\subsection{Relation to Lipschitz lightness and proof of Corollary \ref{cor:bcdim}}
We now work towards the proof of Corollary \ref{cor:bcdim}. First, we observe that, in the bounded turning setting, the property of ``preserving the diameters of continua'' is equivalent to Lipschitz lightness.
\begin{lemma}\label{lem:equiv}
Let $X$ be a bounded turning metric space and $Y$ another metric space. Let $f\colon X \rightarrow Y$. Then the following are equivalent:
\begin{enumerate}[(i)]
\item $f$ is Lipschitz light.
\item $\diam(f(K))\approx \diam(K)$ for every continuum $K\subseteq X$. 
\end{enumerate}
The constants in each item depend only on the constants in the other item and the bounded turning constant of $X$.
\end{lemma}

As noted in the introduction, the implication (i)$\Rightarrow$(ii) in Lemma \ref{lem:equiv} holds without the bounded turning assumption.

\begin{proof}
To prove that (i) implies (ii), suppose that $f$ is Lipschitz light with constant $L$. Let $K$ be a continuum in $X$.  Since $f$ is $L$-Lipschitz, $\diam(f(K))\leq L \diam(K)$. For the other inequality, we observe that since $K$ is connected, it is contained in a $\diam(f(K))$-component of $f^{-1}(K)$. It follows from the Lipschitz lightness of $f$ that
$$ \diam(K) \leq L\diam(f(K)),$$
and so (ii) holds.

Now we prove that (ii) implies (i). Suppose that property (ii) holds for $f$. First of all, $f$ must be Lipschitz: If $x,y\in X$, then by the bounded turning property there is a continuum $K$ containing both points with $\diam(K)\lesssim d(x,y)$. Therefore
$$ d(f(x),f(y)) \leq \diam(f(K)) \lesssim \diam(K) \lesssim d(x,y),$$
and so $f$ is Lipschitz.

Now let $A\subseteq Y$ have $\diam(A)\leq s$ for some $s\geq 0$. Let $P=(x_0, x_1, \dots, x_m)$ be an $s$-path in $f^{-1}(A)$. Join each pair of consecutive points $(x_i,x_{i+1})$ in $P$ by a continuum $K_i$ with 
$$\diam(K_i) \lesssim d(x_i,x_{i+1}) \leq s.$$
Let $K = \cup_{i=0}^{m-1} K_i$, which is a continuum containing $P$. Moreover, $K$ is contained in the closed $bs$-neighborhood of $P$, where $b$ is the bounded turning constant of $X$. Because of this, and the facts that $f$ is Lipschitz and $f(P)\subseteq A$, we have
$$ \diam(f(K)) \lesssim \diam(A) + s \lesssim s.$$
It follows that
$$ \diam(P) \leq \diam(K) \approx \diam(f(K)) \lesssim s.$$
Therefore, $f$ is Lipschitz light.
\end{proof}

\begin{proof}[Proof of Corollary \ref{cor:bcdim}]
Let $X$ be a compact, bounded turning metric space. Let 
$$ n=\dim_N(X)$$
and
$$ a =\inf\{ \dim_A(Y): \text{there is a non-constant branched quasisymmetry of } X \text{ onto } Y\}.$$
Either of these numbers may be infinite.

Assume first that $n<\infty$. Theorem \ref{thm:main} provides a Lipschitz light map from $X^\epsilon$ onto a subset $Y\subseteq \mathbb{R}^n$. By Lemma \ref{lem:equiv}, this map satisfies
$$ \diam(f(K))\approx \diam(K)^\epsilon \text{ for all continua } K\subseteq X,$$
and therefore $f$ is immediately seen to be a branched quasisymmetry. As $\dim_A(Y) \leq \dim_A(\mathbb{R}^n)=n$, this yields
$$ a \leq n.$$
As this inequality trivially holds also when $n=\infty$, we have shown that $a \leq n$ in general.

Now we aim for the reverse inequality: $a \geq n$. We may assume that $a<\infty$, otherwise the inequality is again trivial.

For any $t>0$, the definition of $a$ provides a space $Y$ with $\dim_A(Y)<a+t$ and a branched quasisymmetry $f$ from $X$ onto $Y$. Applying \eqref{eq:factor}, we factor $f$ as 
$$ X \rightarrow (X, d_f) \rightarrow Y,$$
where the first map is a quasisymmetry and the second preserves the diameter of continua. It follows by Lemmas \ref{lem:factor} and \ref{lem:equiv} that the second map is Lipschitz light.

Nagata dimension is a quasisymmetric invariant, so we have $\dim_N(X,d_f)=n$. Lipschitz light maps cannot decrease Nagata dimension (\cite[Lemma 4.1]{GCD}) and Assouad dimension bounds Nagata dimension from above (see \eqref{eq:nagassouad}). So we have
$$ n= \dim_N(X) = \dim_N(X,d_f) \leq \dim_N(Y) \leq \dim_A(Y)<a+t.$$
As $t>0$ was arbitrary, this completes the proof.

\end{proof}

\printbibliography

\end{document}